\documentclass[11pt]{amsart}
\usepackage[T1]{fontenc}
\usepackage[english]{babel}
\usepackage{amscd,amsmath,amsthm,amssymb,graphics}
\usepackage{lmodern,pst-node}
\usepackage{pstcol,pst-plot,pst-3d}
\usepackage{multicol}
\usepackage{epic,eepic}
\usepackage{amsfonts,amssymb,amscd,amsmath,enumitem,verbatim}
\psset{unit=0.7cm,linewidth=0.8pt,arrowsize=2.5pt 4}

\newpsstyle{fatline}{linewidth=1.5pt}
\newpsstyle{fyp}{fillstyle=solid,fillcolor=verylight}
\definecolor{verylight}{gray}{0.97}
\definecolor{light}{gray}{0.9}
\definecolor{medium}{gray}{0.85}
\definecolor{dark}{gray}{0.6}

\def\frk{\frak}               

\def\Phi{{\frk n}}
\def\Phi{{\frk N}}
\def\opn#1#2{\def#1{\operatorname{#2}}} 
\opn\chara{char} \opn\length{\ell} \opn\pd{pd} \opn\rk{rk}
\opn\projdim{proj\,dim} \opn\injdim{inj\,dim} \opn\rank{rank}
\opn\depth{depth} \opn\grade{grade} \opn\height{height}
\opn\embdim{emb\,dim} \opn\codim{codim}

\opn\Tr{Tr} \opn\bigrank{big\,rank}
\opn\superheight{superheight}\opn\lcm{lcm}
\opn\trdeg{tr\,deg}
\opn\reg{reg} \opn\lreg{lreg} \opn\ini{in} \opn\lpd{lpd}
\opn\size{size}\opn\bigsize{bigsize}
\opn\cosize{cosize}\opn\bigcosize{bigcosize}
\opn\sdepth{sdepth}\opn\sreg{sreg}
\opn\link{link}\opn\fdepth{fdepth}
\opn\deg{deg}
\opn\max{max}
\opn\indeg{indeg}
\opn\min{min}
\opn\psln{psln}
\opn\div{div} \opn\Div{Div} \opn\cl{cl} \opn\Cl{Cl}
\let\epsilon\varepsilon
\let\phi=\varphi
\let\kappa=\varkappa

\opn\Spec{Spec} \opn\Supp{Supp} \opn\supp{supp} \opn\Sing{Sing}
\opn\Ass{Ass} \opn\Min{Min}\opn\Mon{Mon} \opn\dstab{dstab} \opn\astab{astab}
\opn\Syz{Syz}
\opn\Ann{Ann} \opn\Rad{Rad} \opn\Soc{Soc}
\opn\Im{Im} \opn\Ker{Ker} \opn\Coker{Coker} \opn\Am{Am}
\opn\Hom{Hom} \opn\Tor{Tor} \opn\Ext{Ext} \opn\End{End}
\opn\Aut{Aut} \opn\id{id}

\opn\nat{nat}
\opn\pff{pf}
\opn\Pf{Pf} \opn\GL{GL} \opn\SL{SL} \opn\mod{mod} \opn\ord{ord}
\opn\Gin{Gin} \opn\Hilb{Hilb}\opn\sort{sort}
\opn\initial{init}
\opn\ende{end}
\opn\height{height}
\opn\depth{depth}
\opn\type{type}
\opn\ldim{ldim}
\opn\lk{lk}
\opn\del{del}

\opn\aff{aff} \opn\con{conv} \opn\relint{relint} \opn\st{st}
\opn\lk{lk} \opn\cn{cn} \opn\core{core} \opn\vol{vol}
\opn\link{link} \opn\star{star}\opn\lex{lex}
\opn\gr{gr}

\def\pot#1#2{#1[\kern-0.28ex[#2]\kern-0.28ex]}

\opn\dirlim{\underrightarrow{\lim}}
\opn\inivlim{\underleftarrow{\lim}}
\def\Implies{\ifmmode\Longrightarrow \else
        \unskip${}\Longrightarrow{}$\ignorespaces\fi}
\def\implies{\ifmmode\Rightarrow \else
        \unskip${}\Rightarrow{}$\ignorespaces\fi}
\def\iff{\ifmmode\Longleftrightarrow \else
        \unskip${}\Longleftrightarrow{}$\ignorespaces\fi}

\let\:=\colon
 \theoremstyle{plain}
\newtheorem{Theorem}{Theorem}[section]
 \newtheorem{Lemma}[Theorem]{Lemma}
 \newtheorem{Corollary}[Theorem]{Corollary}

 \theoremstyle{definition}
 \newtheorem{Definition}[Theorem]{Definition}
 \newtheorem{Remark}[Theorem]{Remark}
 
 \newtheorem{Example}[Theorem]{Example}

\DeclareMathOperator{\g}{\mathcal{G}}

\let\epsilon\varepsilon
\let\kappa=\varkappa
\def\qed{\ifhmode\textqed\fi
      \ifmmode\ifinner\quad\qedsymbol\else\dispqed\fi\fi}
\def\textqed{\unskip\nobreak\penalty50
       \hskip2em\hbox{}\nobreak\hfil\qedsymbol
       \parfillskip=0pt \finalhyphendemerits=0}
\def\dispqed{\rlap{\qquad\qedsymbol}}

\opn\dis{dis}
\def\pnt{{\raise0.5mm\hbox{\large\bf.}}}

\opn\Lex{Lex}

\begin{document}

\author[Mafi, Rasul Qadir and Saremi]{ Amir Mafi, Rando Rasul Qadir and Hero Saremi}
\title{Sequentially Cohen-Macaulay and pretty clean monomial ideals}

\address{Amir Mafi, Department of Mathematics, University Of Kurdistan, P.O. Box: 416, Sanandaj, Iran.}
\email{a\_mafi@ipm.ir}
\address{Rando Rasul Qadir, Department of Mathematics, University of Kurdistan, P.O. Box: 416, Sanandaj,
Iran.}
\email{rando.qadir@univsul.edu.iq}
\address{Hero Saremi, Department of Mathematics, Sanandaj Branch, Islamic Azad University, Sanandaj, Iran.}
\email{h-saremi@iausdj.ac.ir or hero.saremi@gmail.com}

\begin{abstract}
Let $R=K[x_1,\ldots, x_n]$ be the polynomial ring in $n$ variables over a field $K$ and $I$ be monomial ideal of $R$. 
In this paper, we show that if $I$ is a generic monomial ideal, then $R/I$ is pretty clean if and only if $R/I$ is sequentially Cohen-Macaulay. 
Furthermore, we prove that this equivalence remains unchanged for some special monomial ideals. Moreover, we provide an example that disproves the conjecture raised in \cite[p. 123]{S1} regarding generic monomial ideals. 
\end{abstract}


\maketitle
\section*{Introduction}
Throughout this paper, we denote $R=K[x_1,\ldots,x_n]$ a polynomial ring in $n$ variables over a field $K$ and $I$ is a monomial ideal of $R$. We assume, as usual, by $\g(I)$ the unique minimal set of monomial generators of $I$. For each monomial $u=x_1^{a_1} \ldots x_n^{a_n}$, we define the support of $u$ to be $\supp(u) = \{x_i | a_i > 0 \}$.
A monomial ideal
$I=(u_1,\ldots, u_r)$ is called {\it generic} if for each two distinct minimal generators $u_i$ and $u_j$ with the same positive degree in some variable $x_s$, there exists a third generator $u_t$ which is strictly divides $\lcm(u_i,u_j)$, i.e., $\supp(\lcm(u_i,u_j))=\supp(\frac{\lcm(u_i,u_j)}{u_t})$.
The concept of generic monomial ideals was introduced in \cite{MSY}, extending the earlier definition found in \cite{BPS}. For further details, refer also to \cite{Y}.

The cyclic module $R/I$ is defined {\it clean} if there exists a chain of monomial ideals \[\mathcal{F}: I=I_0\subset I_1\subset I_2\subset\ldots\subset I_{r-1}\subset I_r=R\] such that $I_{i}/I_{i-1}\cong R/{\frak{p}_i}$ with $\frak{p}_i$ is a minimal prime ideal of $I$. In other words, for $i=1,\ldots,r$, there exists a monomial element $u_{i}\in I_{i}$ such that $I_{i}=(I_{i-1},u_{i})$ and $\frak{p}_i=(I_{i-1}:u_{i})$. This filtration is called a {\it monomial prime filtration} of $R/I$. The set $\Supp(\mathcal{F})=\{\frak{p}_1,\ldots,\frak{p}_r\}$ is called the support of the prime filtration $\mathcal{F}$. It is known that 
$\Ass(I)\subseteq\Supp(\mathcal{F})\subseteq V(I)$. By applying \cite{D}, for a clean filtration $\mathcal{F}$ of $R/I$, one has $\min\Ass(I)=\Ass(I)=\Supp(\mathcal{F})$. Herzog and Popescu in \cite{HP} defined that the module $R/I$ is {\it pretty clean} when there is a prime filtration 
$\mathcal{F}$ with the following property: if $\frak{p}_i\subset\frak{p}_j$, then $j<i$. It is clear that the clean modules are pretty clean and if $I$ is a squarefree monomial ideal, then $R/I$ is pretty clean if and only if $R/I$ is clean, see \cite[Corollary 3.5]{HP}. However, this equivalence for any monomial ideal is not true, see \cite[Example 3.6]{HP}. Furthermore, if there is a prime filtration $\mathcal{F}$ with $\Ass(I)=\Supp(\mathcal{F})$, then $R/I$ is said to be {\it almost clean}, see \cite{HVZ}. From \cite[Corollary 3.4]{HP},  pretty clean modules are almost clean. In addition, Herzog and Popescu \cite[Corollary 4.3]{HP} showed that if $R/I$ is pretty clean, then $R/I$ is sequentially Cohen-Macaulay.   

Suppose $\Delta$ is a simplicial complex, $I=I_{\Delta}$ is the Stanley-Reisner ideal of $\Delta$ which is generated by all squarefree monomials $x_{i_1}x_{i_2}\ldots x_{i_r}$ such that $\{i_1,\ldots,i_r\}$ is non-face of $\Delta$ and $K[\Delta]=R/I$ is the Stanley-Reisner ring. 
Dress in \cite{D} showed that $\Delta$ is shellable in non-pure case if and only if $K[\Delta]$ is clean.   

In this paper, we show that if $I$ is a generic monomial ideal, then $R/I$ is pretty clean if and only if $R/I$ is sequentially Cohen-Macaulay. In addition, we give this characterization for some special monomial ideals. Furthermore, we  provide an example that disproves the conjecture raised in \cite[p. 123]{S1} regarding generic monomial ideals.

For any unexplained notion or terminology, we refer the reader to \cite{HH} and \cite{Vi}. Several explicit examples were  performed with help of the computer algebra system Macaulay 2 \cite{G}.

\section{Preliminaries}
In this section, we revisit some essential definitions and properties that will be utilized throughout the article.
Let $\Delta$ be a simplicial complex defined over the vertex set $V=\{x_1,\ldots, x_n \}$. The elements of $\Delta$ are referred to as faces, while a facet of $\Delta$ is a maximal face with respect to inclusion. The dimension of a face $F$ is given by $\vert F \vert -1$ and the dimension of a complex $\Delta$ is determined as the maximum dimensions of its facets. If all facets of $\Delta$ have the same dimension, then $\Delta$ is called {\it pure}.
The simplicial complex $\Delta$ is called {\it connected} if there exists a sequence of facets $F_0,F_1,\ldots, F_r$ such that $F_{i-1}\cap F_{i}\neq\emptyset$ for $i=1,\ldots, r$.

For a given simplicial complex $\Delta$ on $V$, we define $\Delta^{\vee}$ by $\Delta^{\vee} = \{V \setminus A~ | ~A \notin \Delta \}$.
The simplicial complex $\Delta^{\vee}$ is called the {\it Alexander dual} of $\Delta$.
For every subset $F \subseteq V$, we set $x_{ F} = \prod_{x_j \in F} x_{j}$. The {\it Stanley-Reisner} ideal of $\Delta$ over $K$ is the ideal $I_{\Delta}$ of $R$ which is generated by those squarefree monomials $x_F$ with $F \notin \Delta$. Let $I$ be an arbitrary squarefree monomial ideal. Then there is a unique simplicial complex $\Delta$ such that $I = I_{\Delta}$. For simplicity,
we often write $I^{\vee}$ to denote the ideal $I_{\Delta^{\vee}}$, and we call it the {\it Alexander dual} of $I$.
If $I$ is a squarefree monomial ideal $I = \cap_{i=1}^{t} \frak{p_{i}}$, where each of the $\frak{p_{i}}$ is a monomial prime ideal of $I$, then the ideal $I^{\vee}$ is minimally generated by monomials $u_i = \prod_{x_{j} \in G(\frak{p_{i}})} x_{j}$.

For a homogeneous ideal $I$, we write $(I_i)$ to denote the ideal generated by the degree $i$ elements of $I$ and if $I$ is generated by squarefree monomials, then we denote by $I_{[i]}$ the ideal generated by squarefree monomials of degree $i$ of $I$. Herzog and Hibi \cite{HH1}
introduced the following definition.

\begin{Definition}
A monomial ideal $I$ is componentwise linear if and only if $(I_i)$ has a linear resolution for all $i$. If $I$ is a squarefree monomials ideal, then $I$ is componentwise linear if and only if $I_{[i]}$ has a linear resolution for all $i$.	
\end{Definition}

Some familiar classes of ideals which are componentwise linear, for example  ideals with linear resolutions, stable ideals, squarefree strongly stable ideals and homogeneous linear quotients ideals (see \cite{HH}).

We recall the following remarkable result, we will incorporate it in the continuation. 
\begin{Theorem}\label{T0}
Let $I=I_{\Delta}$ be a squarefree monomial ideal of $R$. Then the following statements hold:
\begin{itemize}
\item[(i)] $R/I$ is Cohen-Macaulay if and only if the Alexander dual $I^{\vee}$ has a linear resolution, see \cite{ER}.
\item[(ii)] $K[\Delta]$ is clean if and only if the simplicial complex $\Delta$ is shellable if and only if $I^{\vee}$ has a linear quotients, see \cite{D, HHZ1}.
\item[(iii)] $R/I$ is sequentially Cohen-Macaulay if and only if the Alexander dual $I^{\vee}$ is componentwise linear, see \cite{HH1}. 

\end{itemize}
\end{Theorem}

\section{Main Results}
We start this section by the following definition is due to Stanley \cite[Sec. II, 3.9]{S}.

\begin{Definition}\label{D1}
Let $I$ be a monomial ideal of $R$. The module $R/I$ is sequentially Cohen-Macaulay if and only if there exists a chain of monomial ideals 
\[I_0=I\subset I_1\subset I_2\subset\ldots\subset I_r=R,\]
such that each factor module $I_{i}/I_{i-1}$ is Cohen-Macaulay with $\dim I_{i}/I_{i-1}<\dim I_{i+1}/I_{i}$ for $i=1,\ldots,r-1$. Furthermore, if this property is satisfied, then this chain of ideals is uniquely determined.
\end{Definition}

\begin{Example}\label{E1}
Let $I$ be a monomial ideal of $R$. Then $R/I$ is pretty clean, implying that it is sequentially Cohen-Macaulay when at least one of the following conditions holds:
\begin{enumerate}
	
\item[(i)] if $n\leq 3$, see \cite[Theorem 1.10]{S2}. 
\item[(ii)] if $I=(u_1,\ldots,u_s)$ is a squarefree monomial ideal such that either $s\leq 3$ or $\supp(u_i)\cup\supp(u_j)=\{x_1,\ldots,x_n\}$ for all $1\leq i\neq j\leq n$, see \cite[Theorem 2.6]{MNS}. 
\item[(iii)] if $\Ass(I)$ is totally ordered with respect to inclusion, see \cite[Proposition 5.1]{HP} or \cite{CDSS}. 
\end{enumerate}	
\end{Example}

\begin{Lemma}\label{L1}
Let $I$ be a generic monomial ideal. $\Supp(\mathcal{F})=\Ass(I)$.
\end{Lemma}

\begin{proof}
We consider the following prime filtration $\mathcal{F}: I=I_0\subset I_1\subset I_2\subset\ldots\subset I_{r-1}\subset I_r=R$ such that $I_{i}/I_{i-1}\cong R/{\frak{p}_i}$. By applying \cite[Proposition 7.1]{HP}, there is irredundant decomposition into irreducible components $I=\bigcap_{i=1}^rQ_i$. 
Since $I$ is a generic monomial ideal by using \cite[Proposition 2]{A} there is a monomial element $u_j\in\bigcap_{1\leq i\neq j\leq r}Q_i\setminus Q_j$ such that $(I:u_j)=\frak{p_j}$. Therefore $\frak{p_j}\in\Ass(I)$ and we deduce that $\Supp(\mathcal{F})\subseteq\Ass(I)$. The other inclusion is clear and so we have 
 $\Supp(\mathcal{F})=\Ass(I)$, as required.
\end{proof}

By the previous lemma, we can conclude the following result which was appeared in \cite[Corollary 3.4]{S2}.

\begin{Corollary}\label{C1}
Let $I$ be a generic monomial ideal. Then $R/I$ is almost clean. In particular, $R/I$ is clean if and only if $\min\Ass(I)=\Ass(I)$.
\end{Corollary}

In the following result, $R/I$ is regarded as being satisfied Serre's condition $(S_r)$ when $\depth(R/I)_{\frak{p}}\geq\min\{r,\dim(R/I)_{\frak{p}}\}$, for all $\frak{p}\in\Spec(R)$.
\begin{Corollary}\label{C2}
Let $I$ be a generic monomial ideal. Then the following statements are equivalent:
\begin{enumerate}
\item[(i)] $R/I$ is Cohen-Macaulay;
\item[(ii)] $R/I$ satisfies in $(S_r)$ for all $r\geq 1$ ;
\item[(iii)] $R/I$ satisfies in $(S_1)$;
\item[(iv)] $R/I$ is clean.
\end{enumerate}
\end{Corollary}

\begin{proof}
$(i)\Longrightarrow (ii)\Longrightarrow (iii)$ are clear as discussed on page $183$ in \cite{M}.\\
$(iii)\Longrightarrow (iv)$. Since $R/I$ satisfies in $(S_1)$; by details can be found on page $183$ in \cite{M} we obtain $\min\Ass(I)=\Ass(I)$.
Therefore Corollary \ref{C1} fulfilled the results.\\
$(iv)\Longrightarrow (i)$. By applying Corollary \ref{C1} $\min\Ass(I)=\Ass(I)$. Thus \cite[Theorem 2.5]{MSY} yield the result, as required. 
\end{proof}

\begin{Theorem}\label{T1}
Let $I$ be a generic monomial ideal. Then $R/I$ is pretty clean if and only if $R/I$ is sequentially Cohen-Macaulay.
\end{Theorem}

\begin{proof}
$(\Longrightarrow)$. If $R/I$  is pretty clean, then $R/I$ is sequentially Cohen-Macaulay by applying \cite[Corollary 4.3]{HP}.\\
$(\Longleftarrow)$. Suppose $R/I$ is sequentially Cohen-Macaulay. We can consider the following chain of monomial ideals
$I_0=I\subset I_1\subset I_2\subset\ldots\subset I_r=R$ such that all factor module $I_{i}/I_{i-1}$ is Cohen-Macaulay with $\dim I_{i}/I_{i-1}<\dim I_{i+1}/I_{i}$ for $i=1,\ldots,r-1$. Thus, for each $i$, we can pick $\frak{p}_i\in\Ass(I_{i}/I_{i-1})$ such that $\height(\frak{p}_1)>\height(\frak{p}_2)>\ldots>\height(\frak{p}_r)$. By applying \cite[Proposition 2.5]{HP}, it follows that $\Ass(I)=\bigcup_{1\leq i\leq r}\Ass(I_{i}/I_{i-1})$ and also by using Lemma \ref{L1}, $\Supp(\mathcal{F})=\Ass(I)$. Therefore we deduce that $\height(\frak{p}_{i})\leq\height(\frak{p}_{i-1})$
for all $i$ and from \cite[Corollary 2.7]{S1} we conclude that $R/I$ is pretty clean, as required. 
\end{proof}

\begin{Remark}
It is known that another interesting class of monomial ideals, specifically squarefree monomial ideals known as matroidal ideals, satisfies the result of Theorem \ref{T1}. Specifically, if $I$ is a matroidal ideal, then $R/I$ is pretty clean if and only if $R/I$ is sequentially Cohen-Macaulay, as stated in \cite[Corollary 3.9]{HMS}. 
\end{Remark}

In \cite[p.123]{S1}, Soleyman Jahan \textit{conjectured} that all generic monomial ideals are pretty clean. However, the following example demonstrates a counterexample to this conjecture: it is not only not pretty clean, but it also fails to be sequentially Cohen-Macaulay. Furthermore, this example illustrates that almost clean ideals are not necessarily pretty clean in general.

\begin{Example}
Let $n=4$ and $I=(x_1^2x_2,x_2^2x_3,x_1x_4^2,x_3^2x_4)$. Then it is clear that $I$ is generic and $J=\sqrt{I}=(x_1,x_3)(x_2,x_4)$. Thus 
$J^{\vee}=(x_1x_3,x_2x_4)$ and it is not componentwise linear. Therefore from Theorem \ref{T0}(iii), we conclude that 
$J$ is neither sequentially Cohen-Macaulay nor pretty clean. Therefore, by applying \cite[Theorem 2.6]{HTT}, we obtain that 
$I$ is neither sequentially Cohen-Macaulay nor pretty clean.
\end{Example}

In order to continue the paper, we require the following well-known theorem, for which proofs have been established in \cite[Theorem 2.4]{VVW}.

\begin{Theorem}\label{T2}
Let $\Delta$ be a simplicial complex on the vertex set $V=\{x_1,\ldots,x_n\}$. Then the following statements hold:
\begin{enumerate}
\item[(i)] if $\dim\Delta=0$, then $\Delta$ is both pure shellable and Cohen-Macaulay; 
\item[(ii)] if $\dim\Delta=1$, then $\Delta$ is both pure shellable and Cohen-Macaulay if and only if $\Delta$ is connected.
\end{enumerate}
\end{Theorem}

\begin{Remark}
Note that if $I$ is a squarefree monomial ideal of single degree $n-2$, then by applying Theorems \ref{T0}, \ref{T2} we can conclude that 
$I$ has a linear resolution if and only if $I$ has a linear quotients, compare with \cite[Theorem 2]{MS}.
\end{Remark}

\begin{Corollary}\label{C2}
Let $I$ be a squarefree monomial ideal of degree $n-1$. Then $R/I$ and $R/{I^{\vee}}$ are pretty clean. In particular, $R/I$ and  $R/{I^{\vee}}$ are sequentially Cohen-Macaulay.  
\end{Corollary}

\begin{proof}
We can consider the simplicial complex $\Delta$ such that $I=I_{\Delta}$. Since $I$ is of degree $n-1$, it follows $\dim\Delta^{\vee}=0$ and this implies that 
$R/{I^{\vee}}$ is both clean and Cohen-Macaulay, by applying Theorem \ref{T2}. Thus $R/{I^{\vee}}$ is both pretty clean and sequentially Cohen-Macaulay.
For the other part of the result, we may assume that $I=(u_1,\ldots,u_s)$ is a squarefree monomial ideal such that $\supp(u_i)\cup\supp(u_j)=\{x_1,\ldots,x_n\}$ for all $1\leq i\neq j\leq n$. Therefore from \cite[Theorem 2.6]{MNS} we obtain that $R/I$ is both pretty clean and sequentially Cohen-Macaulay. This completes the proof. 
\end{proof}

\begin{Corollary}\label{C3}
Let $\Delta$ be a simplicial complex on the vertex set $V=\{x_1,\ldots,x_n\}$.  If $\dim\Delta=1$, then $\Delta$ is non-pure shellable if and only if $\Delta$ is sequentially Cohen-Macaulay.
\end{Corollary}

\begin{proof}
If $\Delta$ is non-pure shellable, then it is known that $\Delta$ is sequentially Cohen-Macaulay, by \cite[Theorem 6.3.32]{Vi}.
Conversely, suppose that $\Delta$ is sequentially Cohen-Macaulay. Then by using \cite[Theorem 6.3.30]{Vi} we conclude that the pure $i$-skeleton $\Delta^{[i]}$ is Cohen-Macaulay  for $i=-1,0,1$. Therefore,  $\Delta$ may has the following facets $F_1,\ldots, F_r, \{i_1\}, \ldots, \{i_s\}$, which  $F_1,\ldots, F_r$ are facets of $\Delta^{[1]}$ and $\{i_1\}, \ldots, \{i_s\}$ are facets of $\Delta^{[0]}$. From Theorem \ref{T2}, we deduce that  $\Delta^{[1]}$ and $\Delta^{[0]}$
are pure shellable. Thus by applying \cite[p.136]{HH1}, we obtain that $\Delta$ is non-pure shellable, as required. 
\end{proof}

\begin{Corollary}\label{C4}
Let $I$ be a squarefree monomial ideal of degree $\geq n-2$. Then $R/I^{\vee}$ is sequentially Cohen-Macaulay if and only if $R/I^{\vee}$ is pretty clean. 
\end{Corollary}

\begin{proof}
Suppose that $\Delta$ is a simplicial complex with $I=I_{\Delta}$. Since $I$ is of degree $\geq n-2$, this implies that  $\dim\Delta^{\vee}\leq 1$. Therefore, Theorem \ref{T2} and Corollary \ref{C3} fulfilled the result.     
\end{proof}

\begin{Theorem}
Let $n=5$ and $\dim R\leq 5$. If $I$ is squarefree monomial ideal of single degree $d$ with $1\leq d\leq\dim R$, then $R/I^{\vee}$ is sequentially cohen-Macaulay if and only if $R/I^{\vee}$ is pretty clean. 
\end{Theorem}

\begin{proof}
By \cite[Theorem 1.3]{AP}, we may assume that $\dim R=5$. If $d=1,5$, there is nothing to prove. By using Corollary \ref{C4}, we have the result for $d=3,4$. If $d=2$, then by applying \cite[Proposition 2.8]{MNS} and Theorem \ref{T0} the result follows, as required.
\end{proof}

Based on established results, we observe that certain edge ideals are consistently pretty clean, while others are pretty clean if and only if they are sequentially Cohen-Macaulay.
\begin{Remark}
\begin{enumerate}
\item[(i)] If $I$ is an edge ideal of a graph with no chordless cycles of length other than $3$ or $5$, then $R/I$ is pretty clean, \cite[Theorem 1]{W}.
In particular, if $I$ is an edge ideal of a chordal graphs, then $R/I$ is pretty clean and so is sequentially Cohen-Macaulay, see also \cite[Theorem 3.2]{FV}.
\item[(ii)] Let $I$ be an edge ideal of a bipartite graph. Then $R/I$ is sequentially Cohen-Macaulay if and only if $R/I$ is pretty clean, \cite[Theorem 2.10]{V}.  
\end{enumerate}
\end{Remark}

The following example appear in \cite{K}. The following examples demonstrate that sequentially Cohen-Macaulay rings are generally not considered pretty clean rings, even in the case of squarefree monomial ideals of small single degree. 
\begin{Example}
Let $n=11$ and $I$ be an ideal of degree $2$ and generated by\\  $x_1x_4,x_1x_5,x_1x_8,x_1x_9,x_2x_5,x_2x_6,x_2x_8,x_2x_{10},x_2x_{11},x_3x_6,x_3x_7,x_3x_9,x_3x_{10},x_4x_7,x_4x_8,x_4x_{11},\\ x_5x_9,x_5x_{10},x_5x_{11},x_6x_8,x_6x_9,x_6x_{11},x_7x_{10},x_7x_{11},x_9x_{11}$.\\
Then $R/I$ is Cohen-Macaulay and so it is sequentially Cohen-Macaulay. But it is not a pretty clean. 
\end{Example}

The following example appear in \cite[Example 4.2]{HV} and this example was attributed to Terai.
\begin{Example}
Let $n=6$ and $I$ be an ideal of degree $3$ and generated by\\
$x_1x_2x_3,x_1x_2x_6,x_1x_3x_5,x_1x_4x_5,x_1x_4x_6,x_2x_3x_4,x_2x_4x_5,x_2x_5x_6,x_3x_4x_6,x_3x_5x_6.$\\ Then $I=I^{\vee}$ is Cohen-Macaulay and so it is sequentially Cohen-Macaulay. But it is not a pretty clean.
\end{Example}

\subsection*{Acknowledgements} We thank Adam Van Tuyl for providing Corollary \ref{C3}.



\end{document}